\documentclass[12pt]{amsart}
\usepackage{amssymb,times,epsfig}

\makeatletter
\def\@strippedMR{}
\def\@scanforMR#1#2#3\endscan{
  \ifx#1M\ifx#2R\def\@strippedMR{#3}
  \else\def\@strippedMR{#1#2#3}
  \fi\fi}
\renewcommand\MR[1]{\relax\ifhmode\unskip\spacefactor3000 \space\fi
  \@scanforMR#1\endscan
  MR\MRhref{\@strippedMR}{\@strippedMR}}
\makeatother

\newtheorem*{Thm*}{Theorem}
\newtheorem{Thm}{Theorem}
\newtheorem{Cor}[Thm]{Corollary}
\newtheorem{Prop}[Thm]{Proposition}
\newtheorem{Lemma}[Thm]{Lemma}

\theoremstyle{definition}
\newtheorem{Defn}{Definition}

\newtheorem{Remark}{Remark}
\newtheorem{Example}{Example}

\newcommand{\mf}[1]{\mathbb{#1}}
\newcommand{\mc}[1]{\mathcal{#1}}
\newcommand{\mb}[1]{\mathbf{#1}}

\newcommand{\norm}[1]{\left\Vert#1\right\Vert}
\newcommand{\abs}[1]{\left\vert#1\right\vert}

\newcommand{\set}[1]{\left\{#1\right\}}
\newcommand{\ip}[2]{\left \langle #1, #2 \right \rangle}
\newcommand{\state}[1]{\varphi \left[ #1 \right]}

\renewcommand{\phi}{\varphi}

\newcommand{\Span}[1]{\mathrm{Span} \left( #1 \right)}

\newcommand{\stateo}[1]{\varphi_\Omega \left[ #1 \right]}

\newcommand{\br}{\medskip\noindent}

\allowdisplaybreaks[1]

\title{Product-type non-commutative polynomial states}
\author[M.~Anshelevich]{Michael Anshelevich}
\address{Department of Mathematics, Texas A\&M University, College Station, TX 77843-3368}
\email{manshel@math.tamu.edu}
\date{\today}

\begin{document}

\begin{abstract}
In \cite{AnsMonic,AnsBoolean}, we investigated monic multivariate non-commutative orthogonal polynomials, their recursions, states of orthogonality, and corresponding continued fraction expansions. In this note, we collect a number of examples, demonstrating what these general results look like for the most important states on non-commutative polynomials, namely for various product states. In particular, we introduce a notion of a product-type state on polynomials, which covers all the non-commutative universal products and excludes some other familiar non-commutative products, and which guarantees a number of nice properties for the corresponding polynomials.
\end{abstract}

\maketitle

\section{Introduction}

\noindent
The purpose of this note is to describe examples illustrating theorems from \cite{AnsMonic} and \cite{AnsBoolean}. These examples will all be ``product-type'' states on non-commutative polynomials. We first recall the usual notion of a product state.

\br
Let $\mu_1, \mu_2$ be two probability measures on $\mf{R}$ all of whose moments are finite; we identify them with states ($=$ positive linear functionals taking the identity to $1$) on polynomials $\mf{R}[x]$ via
\[
\mu_i[P(x)] = \int_{\mf{R}} P(x) \,d\mu_i(x).
\]
There are many measures on $\mf{R} \times \mf{R}$ with marginals $\mu_1, \mu_2$. Among these, the canonical choice is the product measure $\mu_1 \otimes \mu_2$, corresponding to the state on $\mf{R}[x_1,x_2] = \mf{R}[x_1] \otimes \mf{R}[x_2]$ characterized by the factorization property
\[
(\mu_1 \otimes \mu_2)[P(x_1) Q(x_2)] = \mu_1[P(x_1)] \mu_2[Q(x_2)].
\]
For future reference, we note another factorization property that characterizes the product measure. Namely, let $\set{P^{(i)}_n(x)}$ be the monic orthogonal polynomials for $\mu_i$. Then the monic two-variable polynomials
\begin{equation}
\label{Polynomial-factor-commutative}
P_{n,k}(x_1,x_2) = P^{(1)}_n(x_1) P^{(2)}_k(x_2)
\end{equation}
are precisely the monic orthogonal polynomials for $\mu_1 \otimes \mu_2$.

\br
In this note, we are interested in non-commutative products. That is, given states $\mu_1, \mu_2$ as above, we are interested in canonical ``product-type'' states $\mu_1 \cdot \mu_2$ on the algebra of non-commutative polynomials $\mf{R} \langle x_1, x_2 \rangle = \mf{R}[x_1] \ast \mf{R}[x_2]$ whose restrictions to $\mf{R}[x_1]$, $\mf{R}[x_2]$ are $\mu_1$, $\mu_2$, respectively. One approach is to define canonical products on general, not necessarily polynomial, algebras. This approach was taken by Speicher \cite{SpeUniv} and Ben Ghorbal and Sch{\"u}rmann \cite{Ben-Ghorbal-Independence} and extended by Muraki \cite{Muraki-Quasi-universal,Muraki-Natural-products}. In addition to the usual (tensor) product, they obtained four non-commutative products: the free product \cite{Avitzour,Voi85,VDN}, the Boolean product \cite{Bozejko-Riesz-product,SW97} and the monotone and anti-monotone products \cite{Muraki-Brownian-motion,Franz-Monotone-associative}. In Speicher's approach, these are precisely the only constructions which are associative and universal, in the sense that there are universal polynomials expressing joint moments of elements of the product algebra in terms of individual moments of these elements.

\br
Restricting to polynomial algebras changes the context significantly. One can no longer ask for associativity in a straightforward way, since having a method for defining a product state on $\mf{R}[x_1] \ast \mf{R}[x_2]$ does not tell us how to define a product state on $(\mf{R}[x_1] \ast \mf{R}[x_2]) \ast \mf{R}[x_3]$. Universal formulas also no longer make sense, since for example the property
\[
(\mu_1 \cdot \mu_2) [x_1 x_2 x_1] = \mu_1[x_1^2] \mu_2[x_2]
\]
need not guarantee that
\[
(\mu_1 \cdot \mu_2) [x_1 x_2^2 x_1] = \mu_1[x_1^2] \mu_2[x_2^2].
\]
On the other hand, the canonical grading and basis for polynomial algebras make some constructions nicer; for example, while the Boolean and monotone products are in general only defined for non-unital algebras, there is no difficulty in defining them on (unital) polynomial algebras. Nevertheless, there are too many product-type constructions, for example the $q$-deformed products of \cite{NicaQR} and \cite{AnsQCum}, which, while not being universal \cite{LeeMaaObstruction} are well-defined on polynomials.

\br
As a replacement for the universality restriction, we propose to require the factorization property of orthogonal polynomials analogous to equation~\eqref{Polynomial-factor-commutative}. We will see that all the non-commutative universal products have this property. On the other hand, we will also see in Example~\ref{Example:Counter} that the $q$-deformed products do not. Indeed, none of our products are obtained as deformations. Instead, they are constructed by partial degenerations of the free product.

\section{Generalities on product-type states}

\subsection{Polynomials}
Throughout the paper we consider products of two states $\mu_1$, $\mu_2$ on $\mf{R}[x]$, which for simplicity we take to be faithful. Their orthogonal polynomials $\set{P^{(i)}_n(x_i): i = 1, 2}$ satisfy recursion relations
\begin{equation}
\label{Recursion}
x_i P^{(i)}_n(x_i) = P^{(i)}_{n+1}(x_i) + \beta^{(i)}_n P^{(i)}_n(x_i) + \gamma^{(i)}_n P^{(i)}_{n-1}(x_i).
\end{equation}

\subsection{The free semigroup}
We can identify the elements of the free (non-commutative) semigroup on two generators $\mf{FS}(1,2)$ with multi-indices $\vec{u} = (u(1), u(2), \ldots, u(n))$ or words in the letters $\set{1,2}$, monomials in $\set{x_1, x_2}$, and vertices of the infinite binary tree. The semigroup operation will be denoted by concatenation.

\br
A subset $\Omega \subset \mf{FS}(1,2)$ is \emph{hereditary} if for any $\vec{u} \in \Omega$, every postfix of $\vec{u}$ is also in $\Omega$ (our words are written from the right and incremented on the left), that is, for
\[
\vec{u} = (u(1), u(2), \ldots, u(n)) \in \Omega,
\]
each $(u(k), u(k+1), \ldots, u(n)) \in \Omega$. In the binary tree, a hereditary subset is simply a subtree containing the root.

\subsection{Product-type states}
Denote $\mb{1}^n = (1, 1, \ldots, 1)$ and $\mb{2}^n = (2, 2, \ldots, 2)$ the constant words of length $n$.

\begin{Defn}
\label{Defn:Polynomials}
Let $\Omega \subset \mf{FS}(1,2)$ be a hereditary subset which also has the following two properties.
\begin{enumerate}
\item
$\set{\mb{1}^n, \mb{2}^n: n \geq 1} \subset \Omega$.
\item
If $\vec{u} \in \Omega$, $u(1) = i$, $(j, \vec{u}) \in \Omega$, $j \neq i$, then also $(i, \vec{u}) \in \Omega$.
\end{enumerate}
In the binary tree, the second condition corresponds to the tree containing only vertices of the four (out of the possible six) types in Figure~\ref{Figure:Vertices}.

\begin{figure}[hhh]
\psfig{figure=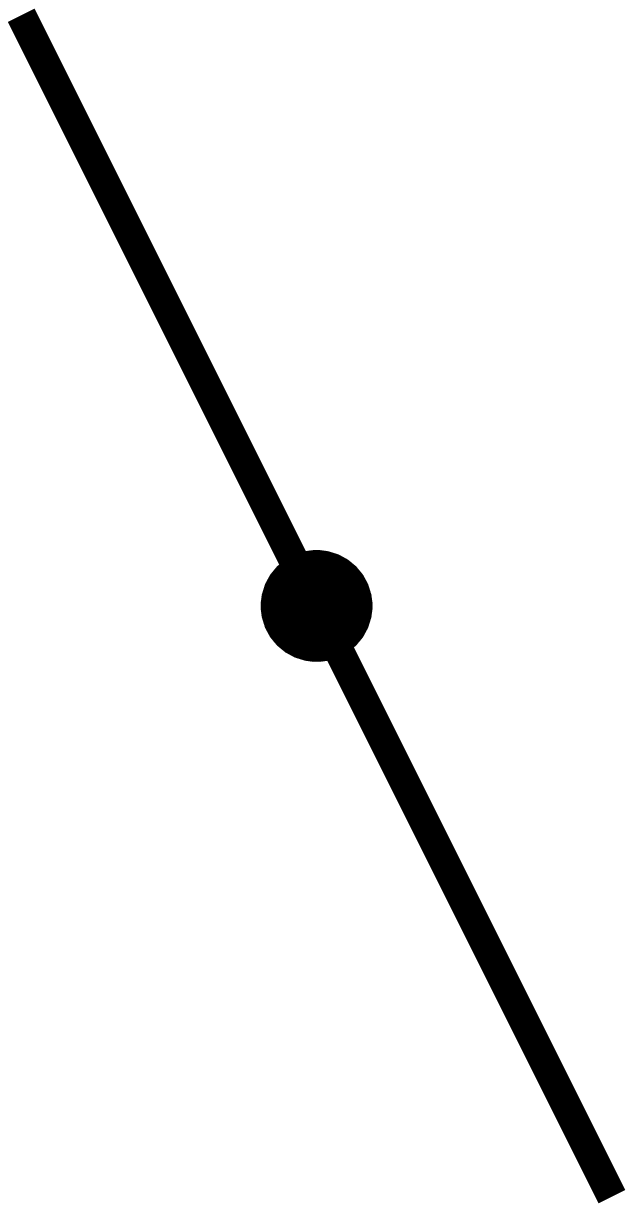,height=0.1\textwidth,width=0.05\textwidth}
\hspace{0.05\textwidth}
\psfig{figure=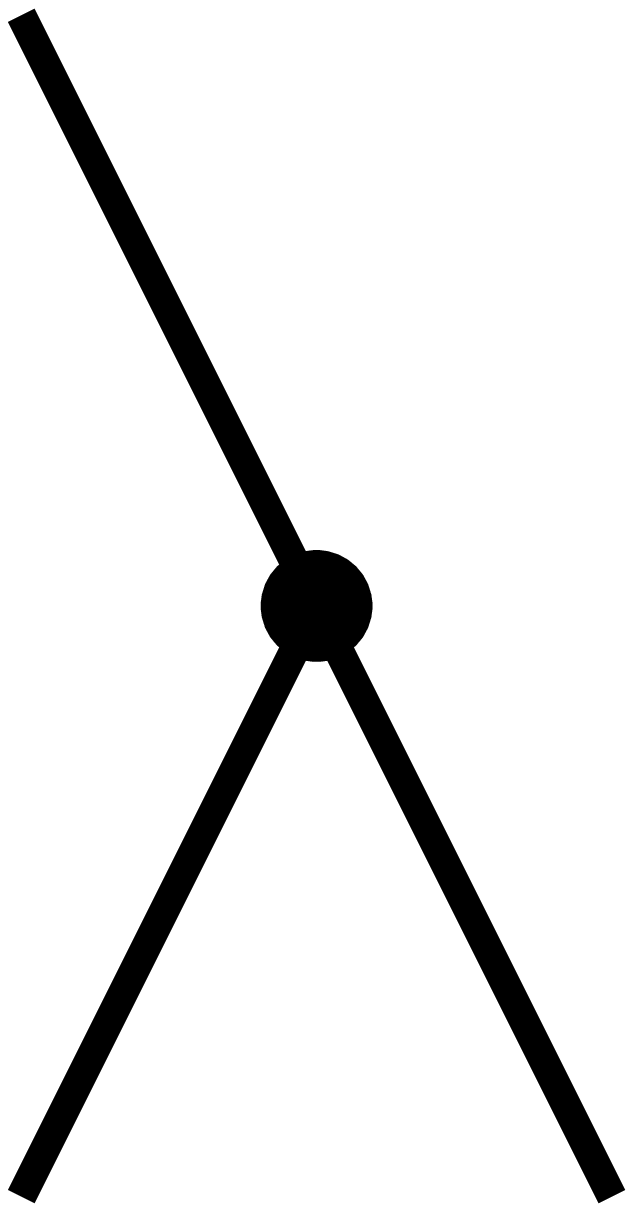,height=0.1\textwidth,width=0.05\textwidth}
\hspace{0.05\textwidth}
\psfig{figure=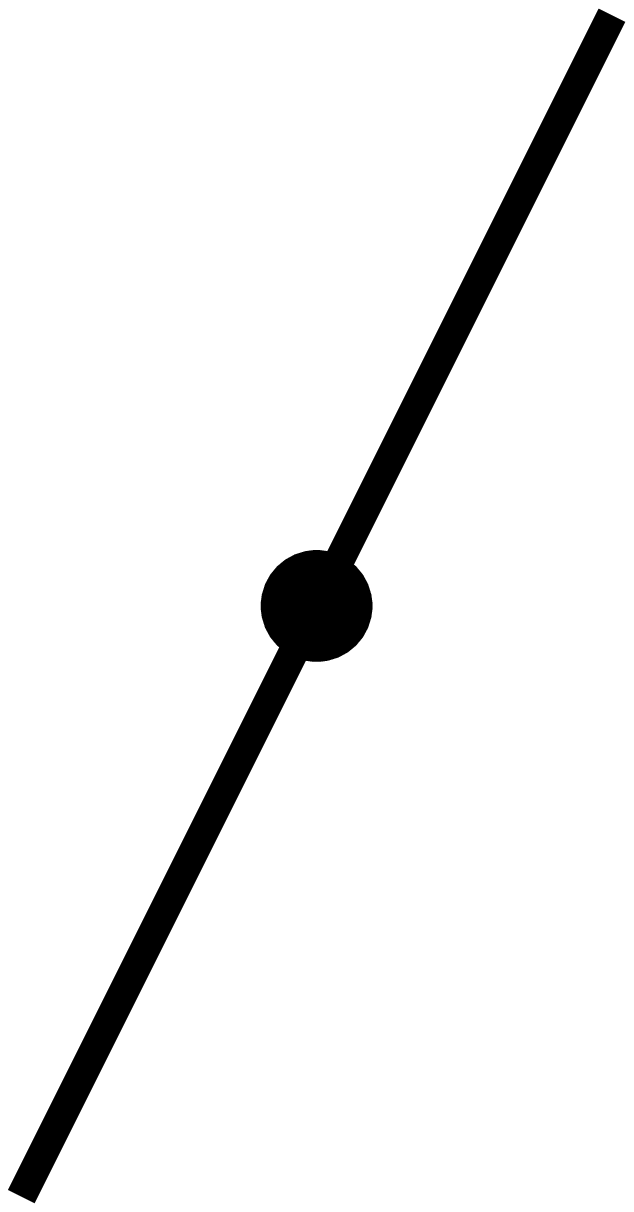,height=0.1\textwidth,width=0.05\textwidth}
\hspace{0.05\textwidth}
\psfig{figure=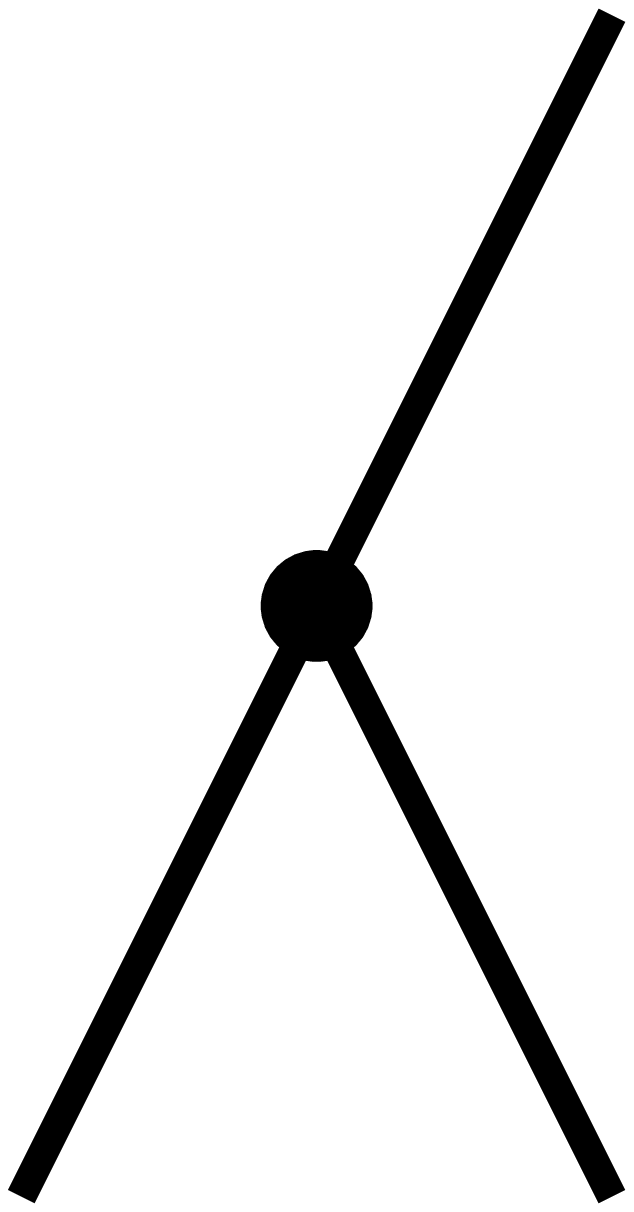,height=0.1\textwidth,width=0.05\textwidth}
\caption{\label{Figure:Vertices}
Vertices appearing in a subtree in Definition~\ref{Defn:Polynomials}.}
\end{figure}

\br
For $\Omega \subset \mf{FS}(1,2)$, denote
\[
\partial \Omega = \set{\vec{u} \in \Omega : u(1) = i, (i, \vec{u}) \not \in \Omega}.
\]
In general, $\partial \Omega$ contains all the leaves of $\Omega$ but may contain other elements as well; for $\Omega$ as above, $\partial \Omega$ consists exactly of its leaves. We will also see that for all the universal products, $\partial \Omega = \emptyset$. Finally, note that if $\Omega$ is hereditary and satisfies the second condition above, then $\Omega \backslash \partial \Omega$ is hereditary as well.

\br
For each $\vec{u} \in \Omega$,
\[
\vec{u} = \mb{1}^{i(1)} \mb{2}^{j(1)} \ldots \mb{1}^{i(n)} \mb{2}^{j(n)},
\]
where $i(1), j(n) \geq 0$ and the rest of $i(k), j(k) \geq 1$, denote
\[
P_{\vec{u}}(x_1, x_2) = \prod_{k=1}^n P^{(1)}_{i(k)}(x_1) P^{(2)}_{j(k)}(x_2).
\]
For $\vec{u} \not \in \Omega$, write $\vec{u} = (\vec{v}, \vec{w})$, where $\vec{w}$ is the longest postfix of $\vec{u}$ in $\Omega$. In this case, denote
\[
P_{\vec{u}}(x_1, x_2) = x_{\vec{v}} P_{\vec{w}}(x_1, x_2),
\]
where
\[
x_{(v(1), v(2), \ldots, v(n))} = x_{v(1)} x_{v(2)} \ldots x_{v(n)}.
\]
\end{Defn}

\begin{Defn}
For $\Omega$ as in Definition~\ref{Defn:Polynomials}, define the linear functional $\phi_\Omega$ on $\mf{R} \langle x_1, x_2 \rangle$ by requiring that
\[
\stateo{1} = 1, \qquad \stateo{P_{\vec{u}}} = 0 \text{ for } \abs{\vec{u}} \geq 1,
\]
so that these polynomials are \emph{centered} with respect to $\phi_\Omega$. We call any functional obtained in this way a \emph{product-type state}.
\end{Defn}

\begin{Prop}
\label{Prop:Orthogonal}
Let $\Omega \subset \mf{FS}(1,2)$ be as in Definition~\ref{Defn:Polynomials}, and $\mu_1, \mu_2$ be faithful.
\begin{enumerate}
\item
Polynomials $\set{P_{\vec{u}} : \vec{u} \in \mf{FS}(1,2)}$ are orthogonal with respect to $\phi_\Omega$. In particular, $\phi_\Omega$ is a positive linear functional.
\item
$\norm{P_{\vec{u}}}_{\phi_\Omega} = 0$ if and only if $\vec{u} \not \in (\Omega \backslash \partial \Omega)$.
\end{enumerate}
\end{Prop}

\br
A direct proof is left to the reader; instead, we will obtain this result below as a corollary of a general theorem.

\begin{Remark}
If $\mu_1, \mu_2$ are not faithful, the proposition still holds with the following modification. If, say, $\mu_1$ is supported on $n$ points, then we require that in $\Omega$, no more than $n$ consecutive $1$'s appear.
\end{Remark}

\begin{Prop}
Any product-type state $\phi_\Omega$ has the property of \emph{stochastic independence}, that is, for any $n, k$,
\[
\stateo{x_1^n x_2^k} = \mu_1[x_1^n] \mu_2[x_2^k].
\]
\end{Prop}

\begin{proof}
First note that
\[
\stateo{x_1^n x_2^k} = \sum_{i=0}^k a_i \stateo{x_1^n P^{(2)}_i(x_2)}
\]
for some $a_i$, with $a_0 = \mu_2[x_2^k]$. Fix $i>0$, and choose $j$ so that $\mb{1}^j \mb{2}^i \in \Omega$, $\mb{1}^{j+1} \mb{2}^i \not \in \Omega$ ($j$ may be zero or infinity). Then for some $b_s$,
\[
\stateo{x_1^n P^{(2)}_i(x_2)} = \sum_{s=0}^j b_s \stateo{P^{(1)}_s(x_1) P^{(2)}_i(x_2)} + \sum_{s = j+1}^n b_s \stateo{x_1^{s-j} P^{(1)}_j(x_1) P^{(2)}_i(x_2)},
\]
which is equal to zero. It follows that
\[
\stateo{x_1^n x_2^k} = a_0 \stateo{x_1^n} = \mu_1[x_1^n] \mu_2[x_2^k]. \qedhere
\]
\end{proof}

\subsection{MOPS}
A product of single-variable monic polynomials is a multivariate monic polynomials. We will see in the proof of Proposition~\ref{Prop:Orthogonal} that in fact, all of our product states have monic orthogonal polynomials (MOPS). Not every state has that property; those that do are characterized in Theorem~2 of \cite{AnsMonic}. Conversely, the following proposition points out general properties of MOPS which served as the starting point for our Definition~\ref{Defn:Polynomials}. Of course, not every state with MOPS is a product-type state.

\begin{Lemma}
Let $\phi$ be a state on $\mf{R} \langle x_1, x_2 \rangle$ with MOPS $\set{Q_{\vec{u}}}$.
\begin{enumerate}
\item
The set $\set{\vec{u} : \norm{Q_{\vec{u}}} \neq 0}$ is a hereditary subset of $\mf{FS}(1,2)$.
\item
$\set{P : \norm{P} = 0} = \Span{\set{Q_{\vec{u}} : \norm{Q_{\vec{u}}} = 0}}$.
\end{enumerate}
\end{Lemma}

\begin{proof}
Part (a) follows from Lemma~2 of \cite{AnsMonic}, and part (b) from Lemma~3 of the same paper.
\end{proof}

\begin{Example}
\label{Example:Counter}
In the next section we describe how all non-commutative universal product fit into our scheme. Here we list two examples which do not.

\br
The tensor product $\phi$ of $\mu_1, \mu_2$ is defined by
\[
\state{x_1^{u(1)} x_2^{v(1)} \ldots x_1^{u(n)} x_2^{v(n)}} = \mu_1 \left[x_1^{u(1) + \ldots + u(n)} \right] \mu_2 \left[x_2^{v(1) + \ldots + v(n)} \right],
\]
and the corresponding orthogonal polynomials are all of the form in equation~\eqref{Polynomial-factor-commutative}. These, however, are not monic orthogonal polynomials in the non-commutative sense. For example, $P_{12}$ and $P_{21}$ are not orthogonal, but rather the same. It is also easy to see that for the tensor product, part (b) of the preceding lemma fails. More generally, commutativity is incompatible with the MOPS condition, and so the tensor product does not fit into our framework.

\br
There are several different notions of the $q$-deformed product. However, many of them coincide in the canonical case of the $q$-product of $q$-Gaussian distributions. In this case $\mu_1 = \mu_2$ are determined by
\[
x P_n(x) = P_{n+1}(x) + \frac{1 - q^{n+1}}{1-q} P_{n-1}(x);
\]
see \cite{NicaQR} or \cite{AnsQCum} for the description of their $q$-deformed product. In particular, some of the monic polynomials obtained by orthogonalization of the monomials are $P_i = x_i$, $P_{12} = x_1 x_2$, $P_{21} = x_2 x_1$, and
\[
P_{121}(x_1, x_2) = x_1 x_2 x_1 - q x_2;
\]
see~\cite{Effros-Popa} for general formulas. First we note that $\ip{P_{12}}{P_{21}} = q \neq 0$ unless $q=0$, so these are not MOPS. Second, we see that $P_{121}$ does not factor for $q \neq 0, \pm 1$.
\end{Example}

\br
By Theorem~2 of \cite{AnsMonic}, every state $\phi$ on $\mf{R} \langle x_1, x_2, \ldots, x_d \rangle$ with MOPS has a representation of a special type on a graded Hilbert space, and such states are parameterized by collections of matrices
\[
\mc{C}^{(k)} = \text{ diagonal non-negative $d^k \times d^k$ matrix, } k = 1, 2, \ldots
\]
and
\[
\mc{T}_i^{(k)} = d^k \times d^k  \text{ matrix, } k = 0, 1, \ldots, i = 1, 2, \ldots, d,
\]
satisfying a commutation relation. For the corresponding state $\phi_{\set{\mc{T}_i}, \mc{C}}$, the entries of these matrices are precisely the coefficients in the recursion relations for the MOPS of $\phi$.

\begin{proof}[Proof of Proposition~\ref{Prop:Orthogonal}]
It is easy to see that if $(i, \vec{u}) \in \Omega$, $\vec{u} = \mb{i}^k \vec{v}$ with $v(1) \neq i$, $k \geq 0$, then
\begin{equation*}
x_i P_{\mb{i}^k \vec{v}} = P_{(i, \vec{u})} + \beta^{(i)}_k P_{\vec{u}} + \gamma^{(i)}_k P_{\mb{i}^{k-1} \vec{v}},
\end{equation*}
where $\gamma_0 \equiv 0$. Also, if $(i, \vec{u}) \not \in \Omega$, then
\begin{equation*}
x_i P_{\vec{u}} = P_{(i, \vec{u})}.
\end{equation*}
So denote, for $\mb{i}^{k+1} \vec{v} \in \Omega$, $v(1) \neq i$, $k \geq 0$,
\[
\mc{T}^{(i)}_{\mb{i}^k \vec{v},\ \mb{i}^k \vec{v}} = \beta^{(i)}_k,
\]
and zero otherwise. Also, for $\mb{i}^{k+1} \vec{v} \in \Omega$, $v(1) \neq i$, $k \geq 1$, which is equivalent to $\mb{i}^{k} \vec{v} \in \Omega \backslash \partial \Omega$, denote
\[
\mc{C}_{\mb{i}^k \vec{v}} = \gamma^{(i)}_k
\]
and zero otherwise. It then follows from Theorem~2 of \cite{AnsMonic} that $\set{P_{\vec{u}}}$ are orthogonal with respect to the state $\phi_{\set{\mc{T}_i}, \mc{C}}$. Since they are centered with respect to $\phi_\Omega$, it follows that they are orthogonal with respect to it, and $\phi_\Omega = \phi_{\set{\mc{T}_i}, \mc{C}}$. Finally, part (b) of the proposition follows from the fact that
\[
\norm{P_{(u(1), u(2), \ldots, u(n))}}^2 = \prod_{i=1}^n \mc{C}_{(u(i), \ldots, u(n))}. \qedhere
\]
\end{proof}

\br
The second condition in Definition~\ref{Defn:Polynomials} is not strictly necessary; the reason for its introduction is the following result.

\begin{Prop}
For generic $\mu_1, \mu_2$, product-type states $\phi_\Omega$ are different for different $\Omega$.
\end{Prop}

\begin{proof}
Suppose $\phi_\Omega = \phi_{\Omega'}$ but $\Omega \neq \Omega'$. By the hereditary condition, there is an $i$ and a $\vec{u} \in \Omega \cap \Omega'$ such that $(i, \vec{u}) \in \Omega$, $(i, \vec{u}) \not \in \Omega'$. Note that the norms induced by $\phi_\Omega, \phi_{\Omega'}$ are the same. By assumption, $\norm{x_i P_{\vec{u}}} = 0$. On the other hand, if $u(1) = i$, then
\[
x_i P_{\vec{u}} = P_{(i, \vec{u})} + \beta^{(i)}_k P_{\vec{u}} + \gamma^{(i)}_k P_{(u(2), u(3), \ldots)}
\]
for some $k$. Since $\set{P_{(i, \vec{u})}, P_{\vec{u}}, P_{(u(2), u(3), \ldots)}}$ are $\phi_\Omega$-orthogonal to each other, and generically $\beta^{(i)}_k, \gamma^{(i)}_k \neq 0$, this implies that all of them, in particular $P_{\vec{u}}$, have norm zero, which contradicts Proposition~\ref{Prop:Orthogonal} since $\vec{u} \in \Omega \backslash \partial \Omega$.

\br
If $u(1) \neq i$, then
\[
x_i P_{\vec{u}} = P_{(i, \vec{u})} + \beta^{(i)}_0 P_{\vec{u}}.
\]
This again implies that $\norm{P_{\vec{u}}} = 0$. However, if $(i, \vec{u}) \in \Omega$, then by the second condition in Definition~\ref{Defn:Polynomials}, also $(u(1), \vec{u}) \in \Omega$, so that $\vec{u} \in \Omega \backslash \partial \Omega$ and we again get a contradiction with Proposition~\ref{Prop:Orthogonal}. 
\end{proof}

\subsection{Continued fractions}
For a state $\phi$ on $\mf{R} \langle x_1, x_2, \ldots, x_d \rangle$, its moment generating function is
\[
M^\phi(z_1, z_2, \ldots, z_d) = 1 + \sum_{i=1}^d \state{x_i} z_i + \sum_{i, j = 1}^d \state{x_i x_j} z_i z_j + \ldots.
\]
It is a classical result that in the one-variable case, such a moment-generating function has a (at least formal) continued fraction expansion
\[
\begin{split}
M^\mu(z)
& = 1 + \mu[x] z + \mu[x^2] z^2 + \mu[x^3] z^3 + \ldots \\
& =
\cfrac{1}{1 - \beta_0 z -
\cfrac{\gamma_1 z^2}{1 - \beta_1 z -
\cfrac{\gamma_2 z^2}{1 - \beta_2 z -
\cfrac{\gamma_3 z^2}{1 - \ldots}}}}
\end{split}
\]
where moreover the coefficients are exactly those in the recursion relation for its monic orthogonal polynomials. Except for this last statement, the continued fraction expansion can be obtained by induction, and the only way it would break down is if some $\gamma_n = 0$, which would indicate that the measure $\mu$ is finitely supported. One may hope that similarly, in the multivariate case one always has a branched continued fraction expansion (see \cite{Skorobogat'ko} and many papers of the same school) such as
\[
\cfrac{1}{1 - \sum_{i_1=1}^d b_{i_1} z_{i_1} - \sum_{j_1, k_1=1}^d
\cfrac{c_{j_1, k_1} z_{j_1} z_{k_1}}{1 - \sum_{i_2=1}^d b_{i_2 i_1} z_{i_2} - \sum_{j_2, k_2=1}^d
\cfrac{c_{j_2 j_1, k_2 k_1} z_{j_2} z_{k_2}}{1 - \ldots}}}
\]
However, such an expansion need not exist in general. Indeed, any power series
\[
1 + \sum a_i z_i + \sum a_{ij} z_i z_j + \ldots
\]
can be written as
\[
\frac{1}{1 - \sum b_i z_i - \sum b_{ij} z_i z_j - \sum b_{ijk} z_i z_j z_k \ldots},
\]
but this need not be equal to some
\[
\frac{1}{1 - \sum b_i z_i - \sum b_{ij} z_i F_{ij} z_j}
\]
if, for example, $b_{ij} = 0$ but $b_{ikj} \neq 0$ for some $k$. Even for a general state with MOPS, such a scalar expansion need not exist; however, one always has a matricial continued fraction.

\begin{Thm*}[Theorem 12 of \cite{AnsBoolean}]
Let $\phi = \phi_{\set{\mc{T}_i}, \mc{C}}$ be a state with MOPS. Then its moment generating function has a matricial continued fraction expansion
\[
\cfrac{1}{1 - \sum_{i_0} z_{i_0} \mc{T}_{i_0}^{(0)} -
\cfrac{\sum_{j_1} z_{j_1} E_{j_1} \mc{C}^{(1)} | \sum_{k_1} E_{k_1} z_{k_1}}{1 - \sum_{i_1} z_{i_1} \mc{T}_{i_1}^{(1)} -
\cfrac{\sum_{j_2} z_{j_2} E_{j_2} \mc{C}^{(2)} | \sum_{k_2} E_{k_2} z_{k_2}}{1 - \sum_{i_2} z_{i_2} \mc{T}_{i_2}^{(2)} -
\cfrac{\sum_{j_3} z_{j_3} E_{j_3} \mc{C}^{(3)} | \sum_{k_3} E_{k_3} z_{k_3}}{1 - \ldots}}}}
\]
\end{Thm*}

\br
Here the vertical bar indicates where to insert the denominator. More precisely, for matrices
\[
A, B \in M_{d^k \times d^k} \simeq M_{d \times d} \otimes M_{d \times d} \otimes \ldots \otimes M_{d \times d},
\]
we use the notation
\[
\frac{E_i A | E_j}{B} = \ip{e_i \otimes I \otimes \ldots \otimes I}{A B^{-1} (e_j \otimes I \otimes \ldots \otimes I)} \in M_{d^{k-1} \times d^{k-1}}.
\]
For example, for $d=2$, $k=2$,
\[
E_1
\begin{pmatrix}
a_{11} & a_{12} & a_{13} & a_{14} \\
a_{21} & a_{22} & a_{23} & a_{24} \\
a_{31} & a_{32} & a_{33} & a_{34} \\
a_{41} & a_{42} & a_{43} & a_{44}
\end{pmatrix}
E_2 =
\begin{pmatrix}
a_{13} & a_{14} \\
a_{23} & a_{24} \\
\end{pmatrix}.
\]

\begin{Cor}
Let $\phi = \phi_{\set{\mc{T}_i}, \mc{C}}$ be a state such that all the matrices $\mc{T}_i^{(k)}$ are diagonal. Denote their entries by $B^{(i)}_{\vec{u}}$, and the entries of (also diagonal) matrices $\mc{C}^{(k)}$ by $\mc{C}_{\vec{u}}$. The moment generating function of $\phi = \phi_{\set{\mc{T}_i}, \mc{C}}$ has a scalar continued fraction expansion
\[
\cfrac{1}{1 - \sum_{i_0} B^{(i_0)}_{\emptyset} z_{i_0} - \sum_{j_1} \mc{C}_{j_1}
\cfrac{z_{j_1} | z_{j_1}}{1 - \sum_{i_1} B^{(i_1)}_{j_1} z_{i_1} - \sum_{j_2} \mc{C}_{j_2 j_1}
\cfrac{z_{j_2} | z_{j_2}}{1 - \sum_{i_2} B^{(i_2)}_{j_2 j_1} z_{i_2} - \sum_{j_3} \mc{C}_{j_3 j_2 j_1}
\cfrac{z_{j_3} | z_{j_3}}{1 - \ldots}}}}
\]
\end{Cor}

\begin{Prop}
If $\phi = \phi_\Omega$ is a product-type state, its moment generating function has a scalar continued fraction expansion corresponding to the subtree $\Omega$ of the binary tree.
\end{Prop}

\begin{proof}
From the proof of Proposition~\ref{Prop:Orthogonal} it follows that $\phi_\Omega = \phi_{\set{\mc{T}_i}, \mc{C}}$, with all $\mc{T}^{(k)}_i$ diagonal. As a result, in the preceding theorem, the continued fraction has the branched form
\[
\cfrac{1}
{1 - \left(\beta^{(1)}_0 z_1 +
\cfrac{\gamma^{(1)}_{1} {z_1 | z_1}}
{1 - \left(\beta^{(1)}_1 z_1 +
\cfrac{\gamma^{(1)}_{2} {z_1 | z_1}}
{1 - \ldots}\right)
- \ldots}\right)
- \left(\beta^{(2)}_0 z_2 +
\cfrac{\gamma^{(2)}_{1} {z_2 | z_2}}
{1 - \left(\beta^{(1)}_0 z_1 +
\cfrac{\gamma^{(1)}_{1} {z_1 | z_1}}
{1 - \ldots}
\right) - \ldots}
\right)}
\]
The branching of the fraction corresponds to the subtree $\Omega$ of the binary tree, and the entry in the fraction corresponding to the word $(\mb{i}^k \vec{v})$ with $k \geq 1$, $v(1) \neq i$ is
\[
\beta^{(i)}_{k-1} z_i + \frac{\gamma^{(i)}_k z_i | z_i}{1 - \ldots}
\]
if $(\mb{i}^k \vec{v}) \in \Omega \backslash \partial \Omega$ and is simply
\[
\beta^{(i)}_{k-1} z_i
\]
if $(\mb{i}^k \vec{v}) \in \partial \Omega$.
\end{proof}

\section{Examples}

\noindent
All the examples in this section are described for $d=2$ for simplicity.

\begin{figure}[hhh]
\psfig{figure=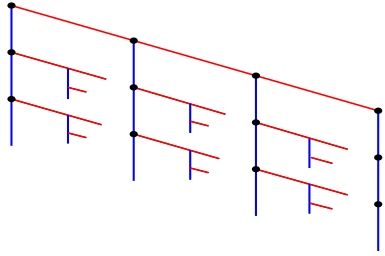,height=0.2\textwidth,width=0.3\textwidth}
\caption{Subtree for the free product}
\end{figure}

\begin{Example}[Free product]
The free product $\phi$ of $\mu_1, \mu_2$ is determined by the condition that if $\set{S_n, R_k: n, k \geq 1}$ are polynomials such that $\mu_1[S_n(x_1)] = 0$ and $\mu_2[R_k(x_1)] = 0$, then $\phi$ is zero on any alternating product of the form
\[
S_1(x_1) R_1(x_2) S_2(x_1) \ldots \text{ or } R_1(x_2) S_1(x_1) R_2(x_2) \ldots.
\]
In this case
\[
\Omega = \mf{FS}(1,2)
\]
and the corresponding polynomials are all alternating products of the form
\[
P^{(1)}_{s(1)}(x_1) P^{(2)}_{t(1)}(x_2) P^{(1)}_{s(2)}(x_1) \ldots P^{(2)}_{t(n)}(x_2),
\]
with $s(2), \ldots, s(n), t(1), \ldots, t(n-1) \geq 1$. Indeed, it follows immediately from the definition of the free product that these polynomials are centered with respect to $\phi$, so $\phi_\Omega = \phi$. The continued fraction for the moment generating function of $\phi$ is
\[
\cfrac{1}
{1 - \beta^{(1)}_0 z_1 -
\cfrac{\gamma^{(1)}_{1} {z_1 | z_1}}
{1 - \beta^{(1)}_1 z_1 -
\cfrac{\gamma^{(1)}_{2} {z_1 | z_1}}
{1 - \ldots}
- \beta^{(2)}_0 z_2 -
\cfrac{\gamma^{(2)}_{1} {z_2 | z_2}}
{1 - \ldots}}
- \beta^{(2)}_0 z_2 -
\cfrac{\gamma^{(2)}_{1} {z_2 | z_2}}
{1 - \beta^{(1)}_0 z_1 -
\cfrac{\gamma^{(1)}_{1} {z_1 | z_1}}
{1 - \ldots}
- \ldots}}
\]
In particular, if all $\beta \equiv 0$, then the continued fraction has a more transparent form
\[
\cfrac{1}
{1 -
\cfrac{\gamma_1^{(1)} {z_1 | z_1}}
{1 -
\cfrac{\gamma_2^{(1)} {z_1 | z_1}}
{1 -
\cfrac{\gamma_3^{(1)} {z_1 | z_1}}
{1 - \ldots}
-
\cfrac{\gamma_1^{(2)} {z_2 | z_2}}
{1 - \ldots}}
-
\cfrac{\gamma_1^{(2)} {z_2 | z_2}}
{1 - \ldots}}
- \cfrac{\gamma_1^{(2)} {z_2 | z_2}}
{1 -
\cfrac{\gamma_1^{(1)} {z_1 | z_1}}
{1 -
\cfrac{\gamma_2^{(1)} {z_1 | z_1}}
{1 - \ldots}
-
\cfrac{\gamma_1^{(2)} {z_2 | z_2}}
{1 - \ldots}}
-
\cfrac{\gamma_2^{(2)} {z_2 | z_2}}
{1 - \ldots}}}
\]
\end{Example}

\begin{figure}[hhh]
\psfig{figure=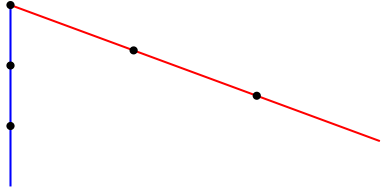,height=0.15\textwidth,width=0.3\textwidth}
\caption{Subtree for the Boolean product}
\end{figure}

\begin{Example}[Boolean product]
The Boolean product $\phi$ of $\mu_1, \mu_2$ is determined by the condition that
\[
\phi[Q(x_1) x_2^{s(1)} x_1^{t(2)} \ldots x_1^{t(n)} R(x_2)]
= \mu_1[Q(x_1)] \mu_2[x_2^{s(1)}] \mu_1[x_1^{t(2)}] \ldots \mu_1[x_1^{t(n)}] \mu_2[R(x_2)],
\]
where all $t(n), s(k) \geq 1$ and $Q, R$ are arbitrary. Note that this is not quite the usual definition of Boolean independence, but it easily seen to be equivalent to it; see \cite{AnsBoolean} or \cite{Popa-multiplicative-boolean}. In this case
\[
\Omega = \set{\mb{1}^n, \mb{2}^n : n \geq 0}
\]
and so the corresponding polynomials are simply
\[
P^{(1)}_k(x_1), \quad P^{(2)}_n(x_2).
\]
For $n \geq 1$,
\[
\phi[x_1^{s(1)} x_2^{t(1)} \ldots x_1^{s(k)} x_2^{t(k)} P^{(1)}_n(x_1)] = 0
\]
since $\mu_1[P^{(1)}_n(x_1)] = 0$, and the same property holds for polynomials ending in $P^{(2)}_n(x_2)$, so it follows that these polynomials are centered with respect to $\phi$ and $\phi = \phi_\Omega$. The continued fraction for the moment generating function of $\phi$ is simply
\[
\cfrac{1}
{1 - \beta^{(1)}_0 z_1 -
\cfrac{\gamma_1^{(1)} {z_1^2}}
{1 - \beta^{(1)}_1 z_1 -
\cfrac{\gamma_2^{(1)} {z_1^2}}
{1 - \beta^{(1)}_2 z_1 -
\cfrac{\gamma_3^{(1)} {z_1^2}}
{1 - \ldots}}}
- \beta^{(2)}_0 z_2 -
\cfrac{\gamma_1^{(2)} {z_2^2}}
{1 - \beta^{(2)}_1 z_2 -
\cfrac{\gamma_2^{(2)} {z_2^2}}
{1 - \beta^{(2)}_2 z_2 -
\cfrac{\gamma_3^{(2)} {z_2^2}}
{1 - \ldots}}}}
\]
\end{Example}

\begin{figure}[hhh]
\psfig{figure=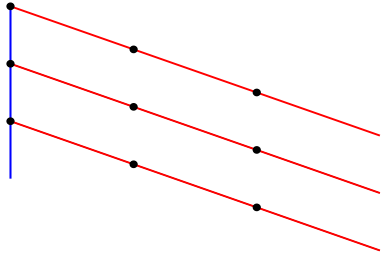,height=0.2\textwidth,width=0.3\textwidth}
\caption{Subtree for the monotone product}
\end{figure}

\begin{Example}[Monotone product]
$\phi$ is determined by the condition that
\[
\phi[Q(x_1) x_2^{s(1)} x_1^{t(2)} x_2^{s(2)} \ldots x_1^{t(n)} R(x_2)]
= \mu_1[Q(x_1) x_1^{t(2)} \ldots x_1^{t(n)}] \mu_2[x_2^{s(1)}] \mu_2[x_2^{s(2)}] \ldots \mu_2[R(x_2)],
\]
where all $t(i), s(j) \geq 1$ and $Q, R$ are arbitrary. Again this is not quite the usual definition of monotone independence, but is easily seen to be equivalent to it. In this case
\[
\Omega = \set{\mb{2}^k \mb{1}^n: k, n \geq 0}
\]
and the corresponding polynomials are products
\[
P^{(2)}_k(x_2) P^{(1)}_n(x_1).
\]
If $k \geq 1$, let $P^{(2)}_k(x_2) = \sum_{i=0}^k a_i x_2^i$. Then
\[
\begin{split}
\phi[x_1^{s(1)} x_2^{t(1)} \ldots x_1^{s(j)} P^{(2)}_k(x_2) P^{(1)}_n(x_1)]
& = \sum_{i=0}^k a_i \phi[x_1^{s(1)} x_2^{t(1)} \ldots x_1^{s(j)} x_2^i P^{(1)}_n(x_1)] \\
& = \sum_{i=0}^k a_i \mu_2[x_2^i] \phi[x_1^{s(1)} x_2^{t(1)} \ldots x_1^{s(j)} P^{(1)}_n(x_1)] \\
& = \mu_2[P^{(2)}_k(x_2)] \phi[x_1^{s(1)} x_2^{t(1)} \ldots x_1^{s(j)} P^{(1)}_n(x_1)] = 0
\end{split}
\]
since $\mu_2[P^{(2)}_k(x_2)] = 0$. It follows that all $\set{P_{\vec{u}}}$ are centered for $\phi = \phi_\Omega$. The continued fraction for the moment generating function of $\phi$, where for clarity we set all $\beta \equiv 0$, is
\[
\cfrac{1}
{1
- \cfrac{\gamma^{(1)}_1 z_1 | z_1}
{1
- \cfrac{\gamma^{(1)}_2 z_1 | z_1}
{1
- \cfrac{\gamma^{(1)}_3 z_1 | z_1}
{1
- \cfrac{\gamma^{(1)}_4 z_1 | z_1}
{1 - \ldots}
- \cfrac{\gamma^{(2)}_1 z_2^2}
{1 - \ldots}}
- \cfrac{\gamma^{(2)}_1 z_2^2}
{1 - \cfrac{\gamma^{(2)}_2 z_2^2}
{1 - \ldots}}}
- \cfrac{\gamma^{(2)}_1 z_2^2}
{1 - \cfrac{\gamma^{(2)}_2 z_2^2}
{1 - \cfrac{\gamma^{(2)}_3 z_2^2}
{1 - \ldots}}}}
- \cfrac{\gamma^{(2)}_1 z_2^2}
{1 - \cfrac{\gamma^{(2)}_2 z_2^2}
{1 - \cfrac{\gamma^{(2)}_3 z_2^2}
{1 - \cfrac{\gamma^{(2)}_4 z_2^2}
{1 - \ldots}}}}}
\]
Anti-monotone product looks very similar, with $1$ and $2$, and right and left, switched.
\end{Example}

\begin{Example}[c-free product]
The c-free product \cite{BLS96}, also known as two-state free product, does not quite fit into our scheme, since in this case we start with two pairs of states, $(\mu_i, \nu_i)$, $i = 1,2$. Nevertheless, it also has the product-type property, as we now explain. Two pairs of states have two pairs of families of orthogonal polynomials
\[
\set{P^{(i)}_n(x_i), Q^{(i)}_k(x_i)}
\]
orthogonal with respect to $\mu_i$, respectively, $\nu_i$, with recursion relations \eqref{Recursion} and
\begin{align*}
x_i Q^{(i)}_n(x_i) & = Q^{(i)}_{n+1}(x_i) + b^{(i)}_n Q^{(i)}_n(x) + c^{(i)}_n Q^{(i)}_{n-1}(x_i).
\end{align*}
The c-free product of these pairs of states is the pair $(\phi, \psi)$, where $\psi$ is the free product $\nu_1 \ast \nu_2$ and $\phi$ is determined by the condition that whenever $\set{S_j, R_j: 1 \leq j \leq n}$ are polynomials such that $\nu_1[S_j(x_1)] = 0$ for $j \geq 2$ and $\nu_2[R_j(x_1)] = 0$ for $j \leq n-1$, then
\begin{multline*}
\phi[S_1(x_1) R_1(x_2) S_2(x_1) \ldots S_n(x_1) R_n(x_2)] \\
= \mu_1[S_1(x_1)] \mu_2[R_1(x_2)] \mu_1[S_2(x_1)] \ldots \mu_1[S_n(x_1)] \mu_2[R_n(x_2)]
\end{multline*}
Again this is not quite the usual definition of c-free independence, so see Lemma~1 of \cite{AnsAppell3}. In this case the orthogonal polynomials  with respect to $\phi$ are alternating products
\[
Q^{(1)}_{s(1)}(x_1) Q^{(2)}_{t(1)}(x_2) Q^{(1)}_{s(2)}(x_1) \ldots Q^{(1)}_{s(n)}(x_1) P^{(2)}_{t(n)}(x_2),
\]
with $s(2), \ldots, s(n), t(1), \ldots, t(n) \geq 1$, or of the same form with $1, 2$ interchanged.  The centeredness, and so orthogonality, of these polynomials with respect to $\phi$ follows directly from the c-free property above, since $\nu_i[Q^{(i)}_s(x_i)] = 0$ and $\mu_2[P^{(2)}_{t(n)}(x_2)] = 0$. The continued fraction, again for all $\beta \equiv b \equiv 0$, is
\[
\cfrac{1}
{1 -
\cfrac{\gamma_1^{(1)} {z_1 | z_1}}
{1 -
\cfrac{\gamma_2^{(1)} {z_1 | z_1}}
{1 -
\cfrac{\gamma_3^{(1)} {z_1 | z_1}}
{1 - \ldots}
-
\cfrac{c_1^{(2)} {z_2 | z_2}}
{1 - \ldots}}
-
\cfrac{c_1^{(2)} {z_2 | z_2}}
{1 - \ldots}}
- \cfrac{\gamma_1^{(2)} {z_2 | z_2}}
{1 -
\cfrac{c_1^{(1)} {z_1 | z_1}}
{1 -
\cfrac{c_2^{(1)} {z_1 | z_1}}
{1 - \ldots}
-
\cfrac{c_1^{(2)} {z_2 | z_2}}
{1 - \ldots}}
-
\cfrac{\gamma_2^{(2)} {z_2 | z_2}}
{1 - \ldots}}}
\]
By looking at the orthogonal polynomials, or at the continued fraction, we note that
\begin{enumerate}
\item
If both $\nu_i = \mu_i$, so that $b^{(i)}_n = \beta^{(i)}_n$, $c^{(i)}_n = \gamma^{(i)}_n$, and $Q^{(i)}_n = P^{(i)}_n$, then $\phi$ is the free product of $\mu_1$ and $\mu_2$.
\item
If $\nu_1 = \nu_2 = \delta_0$, so that $b^{(i)}_n = c^{(i)}_n = 0$ for all $n$ and $Q^{(i)}_n(x_i) = x_i^n$, then $\phi$ is the Boolean product of $\mu_1$ and $\mu_2$.
\item
If $\nu_1 = \delta_1$, $\nu_2 = \mu_2$, then $\phi$ is the monotone product of $\mu_1$ and $\mu_2$ \cite{Franz-Multiplicative-monotone}, while for $\nu_1 = \mu_1$, $\nu_2 = \delta_0$ we get the anti-monotone product.
\end{enumerate}
\end{Example}

\section{Restrictions on states and Hilbert space products}

\begin{Remark}
A weak replacement for associativity of the product in the sense of \cite{SpeUniv} is the following requirement for $\Omega$. Let $\vec{u} \in \mf{FS}(1,2,3)$. It can be written in the form
\[
\vec{u} = (\vec{w}_1 \mb{3}^{i(1)} \vec{w}_2 \ldots \mb{3}^{i(n)} \vec{w}_{n+1}),
\]
with all $\vec{w}_j \in \mf{FS}(1,2)$. We say that $\vec{u} \in \Omega^2$ if each $\vec{w}_j \in \Omega$ and
\[
\mb{1}^{\abs{\vec{w}_1}} \mb{2}^{i(1)} \mb{1}^{\abs{\vec{w}_2}} \ldots \mb{2}^{i(n)} \mb{1}^{\abs{\vec{w}_{n+1}}} \in \Omega.
\]
We say that $\Omega$ is associative if $\Omega^2$ also consists of all
\[
\vec{u} = (\vec{w}_1 \mb{1}^{i(1)} \vec{w}_2 \ldots \mb{1}^{i(n)} \vec{w}_{n+1})
\]
such that each $\vec{w}_j \in \Omega(2,3)$ (defined in the obvious way) and
\[
\mb{2}^{\abs{\vec{w}_1}} \mb{1}^{i(1)} \mb{2}^{\abs{\vec{w}_2}} \ldots \mb{1}^{i(n)} \mb{2}^{\abs{\vec{w}_{n+1}}} \in \Omega.
\]
It is easy to see that all of the universal products satisfy this condition. However, there are many more such sets $\Omega$. One example follows.
\end{Remark}

\begin{figure}[hhh]
\psfig{figure=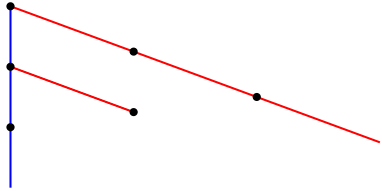,height=0.15\textwidth,width=0.3\textwidth}
\caption{Subtree for the product in Example~\ref{Example:One-branch}}
\end{figure}

\begin{Example}
\label{Example:One-branch}
Let
\[
\Omega = \set{(2,1)} \cup \set{\mb{2}^n, \mb{1}^n: n \geq 0},
\]
so that
\[
\Omega^2 = \set{(2,1), (3,1), (3,2)} \cup \set{\mb{1}^n, \mb{2}^n, \mb{3}^n: n \geq 0}
\]
and $\Omega$ is associative. The corresponding orthogonal polynomials are
\[
P^{(1)}_k(x_1), \quad P^{(2)}_n(x_2), \quad P^{(2)}_1(x_2) P^{(1)}_1(x_1).
\]
One can check that in this case the product state $\phi_\Omega$ satisfies (and is determined by) factorization properties
\[
\phi[\ldots x_2^k x_1^n] = \phi[\ldots x_2^k] \mu_1[x_1^n]
\]
if at least one of $k, n \geq 2$,
\[
\phi[\ldots x_1^n x_2^k] = \phi[\ldots x_1^n] \mu_2[x_2^k]
\]
(which are easy to show) but
\[
\begin{split}
\phi[\ldots x_1^n x_2 x_1]
& = \phi[\ldots x_1^n (P^{(2)}_1(x_2) + \mu_2[x_2])(P^{(1)}_1(x_1) + \mu_1[x_1])] \\
& = \phi[\ldots x_1^n P^{(2)}_1(x_2)] \mu_1[x_1] + \phi[\ldots x_1^n x_1] \mu_2[x_2]
= \phi[\ldots x_1^n x_1] \mu_2[x_2].
\end{split}
\]
Finally, the continued fraction for the moment generating function of $\phi_\Omega$ is
\[
\cfrac{1}
{1 - \beta^{(1)}_0 z_1 -
\cfrac{\gamma_1^{(1)} {z_1 | z_1}}
{1 - \beta^{(1)}_1 z_1 -
\cfrac{\gamma_2^{(1)} {z_1^2}}
{1 - \ldots}
- \beta^{(2)}_0 z_2}
- \beta^{(2)}_0 z_2 -
\cfrac{\gamma_1^{(2)} {z_2^2}}
{1 - \beta^{(2)}_1 z_2 -
\cfrac{\gamma_2^{(2)} {z_2^2}}
{1 - \ldots}}}
\]
Note that if $\beta^{(2)}_0 = 0$ (i.e.\ if $\mu_2[x] = 0$), $\phi_\Omega$ is the same as for the Boolean product.
\end{Example}

\begin{Remark}[Products of Hilbert spaces]
Let $\mc{H}_i = \mf{C} \xi_i \oplus \mc{H}_i^\circ$, $i=1, 2$ be separable Hilbert spaces such that $\mc{H}_i^\circ$ comes with a given orthonormal basis $\set{e^{(i)}_j}$. For any $\Omega$, we can form the product of these spaces $\mc{H}_1 \ast_\Omega \mc{H}_2$ to be the Hilbert space with the orthonormal basis
\[
\set{\xi, e^{(1)}_{u(1)} \otimes e^{(2)}_{v(1)} \otimes \ldots e^{(1)}_{u(n)} \otimes e^{(2)}_{v(n)} : \mb{1}^{u(1)} \mb{2}^{v(1)} \ldots \mb{1}^{u(n)} \mb{2}^{v(n)} \in \Omega}.
\]
In general this product will depend on the choice of the bases, however for special $\Omega$ it may not. Note also that associativity of $\Omega$ is equivalent to the associativity of the corresponding Hilbert space product. 
\end{Remark}

\begin{Prop}
Let $\Omega$ satisfy the conditions of Definition~\ref{Defn:Polynomials}, be associative, and such that the corresponding Hilbert space product is basis-independent. Then $\Omega$ corresponds to one of four non-commutative universal products.
\end{Prop}

\begin{proof}
First note that basis independence allows us to replace any vector $e^{(i)}_j$ with any other $e^{(i)}_k$. This implies that if $\Omega$ contains a word with a consecutive sequence of $i$'s of a certain length, then we can simultaneously replace all sequences of $i$'s of this length in all the words in $\Omega$ by sequences of any other length.

\br
By definition $\Omega$ always contains all $\mb{1}^n, \mb{2}^k$. If it consists only of these sequences, $\phi_\Omega$ is the Boolean product. Otherwise, suppose $\Omega$ it contains one of, hence all, sequences of the form $\mb{2}^k \mb{1}^n$. If it consists only of these sequences, $\phi_\Omega$ is the monotone (or, with $1, 2$ switched, anti-monotone) product. Otherwise, $\Omega$ contains a sequence of the form $\mb{1}^m \mb{2}^k \mb{1}^n$. $k$ is arbitrary. If $n \neq m$, they can be taken to be arbitrary as well. If $n = m > 1$, then by the hereditary property, $\mb{1}^{m-1} \mb{2}^k \mb{1}^n \in \Omega$ and so $m, n$ are again arbitrary. Finally, if $n = m = 1$, $k > 1$, then by associativity
\[
(1, \mb{2}^k, 1), (\mb{2}^{k-1}, 1) \in \Omega \Rightarrow (1, \mb{3}^{k-1}, 2, 1) \in \Omega^2 \Rightarrow (1, \mb{2}^{k-1}, 1, 1) \in \Omega
\]
for which $m=1, n=2$.

\br
Next, we note that any $\mb{1}^{m} \mb{3}^k \mb{1}^i \mb{2}^j \mb{1}^n \in \Omega^2$, and by associativity, $\mb{1}^{m} \mb{2}^k \mb{1}^i \mb{2}^j \mb{1}^n \in \Omega$. Proceeding in this way, we see that any $\vec{u}$ with the rightmost entry $u(n) = 1$ is in $\Omega$. The set of sequences with this property is not associative, since
\[
(1, 2, 2, 2, 1), (1, 2, 1) \in \Omega \Rightarrow (1, 2, 3, 2, 1) \in \Omega^2 \Rightarrow (1, 1, 2, 1, 1), (1, 2), (2,1) \in \Omega.
\]
It follows that $\Omega = \mf{FS}(1,2)$ and $\phi_\Omega$ is the free product.
\end{proof}

\br
\textbf{Acknowledgements.} I would like to thank the organizers of the Workshop on Non-commutative Harmonic Analysis with Applications to Probability for an exciting and enjoyable meeting. I am also grateful to the organizers of the Workshop on Special Functions and Orthogonal Polynomials at FoCM'08, which I attended and which influenced the writing of this note. In particular, I would like to thank Alexander Aptekarev for bringing reference \cite{Skorobogat'ko} to my attention.


\def\cprime{$'$}
\providecommand{\bysame}{\leavevmode\hbox to3em{\hrulefill}\thinspace}
\providecommand{\MR}{\relax\ifhmode\unskip\space\fi MR }
\providecommand{\MRhref}[2]{%
  \href{http://www.ams.org/mathscinet-getitem?mr=#1}{#2}
}
\providecommand{\href}[2]{#2}

\end{document}